\documentclass[10pt]{article}
 
\usepackage[utf8]{inputenc}
\usepackage[T1]{fontenc}
\usepackage{authblk}
\usepackage{tikz-cd}

\usepackage{amsfonts}
\usepackage{amsmath,amssymb}
\usepackage{amscd}
\usepackage{amsthm}
\usepackage{bm}
\usepackage{bbm}
\usepackage{mathtools}
\usepackage{todonotes}
\newtheorem{df}{Definition}[section]
\newtheorem{thm}{Theorem}[section]
\newtheorem{lem}{Lemma}[section]

\newtheorem{prop}{Proposition}[section]
\newtheorem{rem}{Remark}[section]

\newtheorem{que}{Question}[section]

\usepackage{amsthm}

\newcommand{\holie}{L_\infty/A}

\newcommand{\Tsmooth}{T_{\mathcal{O}_M}}
\newcommand*\cocolon{%
        \nobreak
        \mskip6mu plus1mu
        \mathpunct{}%
        \nonscript
        \mkern-\thinmuskip
        {:}%
        \mskip2mu
        \relax
}
\newcounter{para}
\newcommand\mypara{\par\refstepcounter{para}\bigskip\noindent\textbf{\thesection.\thepara}\space}

\title{Homotopy theory of singular foliations}

\author[$\dagger$]{Yaël Frégier}
\author[$\dagger$,$\star$]{Rígel A. Juárez-Ojeda\thanks{Supported throughout the realization of the present work by the fellowship 550849 from CONACYT.}}
\affil[$\dagger$]{LML, Université d'Artois}
\affil[$\star$]{IMJ, Sorbonne Université}
\affil[ ]{\texttt {rigel.juarez@imj-prg.fr, yael.fregier@gmail.com}}
\date{\today}
\date{} % delete this line to display the current date

\begin{document}

\maketitle
\begin{abstract}
    In this article we apply ideas from homotopy theory to the study of singular foliations. We verify that a technical lemma remains valid for left semi-model categories. When applied to the category of $L_\infty$-algebroids thanks to \cite{Nuiten}, this lemma enables to recover results very similar to those of \cite{Lavau} about existence and (up to homotopy) uniqueness of universal $L_\infty$-algebroids associated to a singular foliation. We conclude with some open questions.
\end{abstract}

\section{Introduction}

 Homotopy theory, originally developed in the realm of topology, is now used in most domains of mathematics. One can cite algebraic topology, $\infty$-algebras, operads, higher category theory, mathematical physics or derived geometry as examples. This tendency is in part due to the success of the concept of {\it closed model category} introduced by Quillen in \cite[Chapter 1]{Quillen}. Indeed, a model structure on a category gives access to topological intuition and a collection of important results.

The aim of this article is to illustrate how homotopy theory can be used to get a better understanding of {\it  singular foliations}. Our motivation comes from recent results of Lavau, Laurent-Gengoux and Strobl \cite{Lavau} on singular foliations which are reminiscent of classical results in homotopy theory. We show that these results can actually be deduced form  a {\it left semi-model} structure on the category $\holie$ of {\it $L_\infty$-algebroids}.

We first recall the results of \cite{Lavau} about resolutions of singular foliations and state analogous statements in the language of model categories. We then outline the structure of the paper.

\mypara Foliations are classical objects of study in mathematics. Examples arise in PDEs, Poisson geometry, Lie groups or differential geometry to cite a few. Roughly speaking, a {\it foliation} $\mathcal{F}$ consists in a partition of a smooth manifold $M$ into a collection of sub-manifolds called the {\it leaves} of $\mathcal{F}$. The collection $D(\mathcal{F})$ of tangent spaces to the leaves of $\mathcal{F}$ is called the {\it distribution} associated to the foliation $\mathcal{F}$. The foliation is called {\it regular} when $D(\mathcal{F})$ forms a vector bundle. As the leaves are submanifolds, the module $\Gamma (D(\mathcal{F}))$ of sections of this bundle is involutive under the bracket of vector fields. However, in the general case, i.e. when the tangent spaces to the leaves do not form a vector bundle, one can still consider the modules of vectors fields tangent to the leaves.  {\it Singular foliations} are therefore defined in \cite{An-Sk09}
as locally finitely generated submodules of the module of vector fields, closed under the bracket.

Contrary to regular foliations, which are by now well understood, singular foliations are still mysterious objects. A promising approach has been proposed in \cite{Lavau} to understand a given singular foliation $\mathcal{F}$. The idea is to associate to $\mathcal{F}$ a {\it Lie $\infty$-algebroid} $Q\mathcal{F}$ which is equivalent to $\mathcal{F}$ \cite[Theorem 1.6]{Lavau}. This $Q\mathcal{F}$ is called the {\it universal Lie $\infty$-algebroid over $\mathcal{F}$}. The use of the adjective {\it universal} is justified by \cite[Theorem 1.8]{Lavau} which states that two such universal Lie $\infty$-algebroids over $\mathcal{F}$ are quasi-isomorphic, and that such a quasi-isomorphism is essentially unique (up to homotopy).

\mypara These results should immediately ring a bell to a reader familiar with model categories. For convenience of the reader who is not, let us briefly explain why. Model categories were introduced in order to write down a minimal set of axioms enabling to perform constructions familiar in homological algebra (projective resolutions, derived functors) and in homotopy theory (homotopy relation for maps, fibrations, cofibrations, homotopy category).

Among these constructions, the {\it cofibrant replacement}, which subsumes the notion of projective resolution, is central.  The idea is, given $\mathcal{F}$, to find another object $Q\mathcal{F}$ which is better behaved (i.e. is cofibrant) and remains equivalent to $\mathcal{F}$.
The proper way to formulate these properties is to introduce three classes of maps, $\mathcal{W}$, $Cof$ and $Fib$, respectivelly called {\it weak equivalences}, {\it cofibrations} and {\it fibrations}. $Q\mathcal{F}$ is {\it weak equivalent} to $\mathcal{F}$ means that there exists a map in $\mathcal{W}$ between $Q\mathcal{F}$ and $\mathcal{F}$. $Q\mathcal{F}$ is {\it cofibrant} means that the map  $0\hookrightarrow Q\mathcal{F}$ from the initial object is a cofibration.

Cofibrant replacements are important to derive functors. The {\it left derived functor $D(F)$} of a functor $F$ is defined by $D(F)(\mathcal{F}):=F(Q\mathcal{F})$. A non trivial part of the theory is to ensure that this derived functor is well defined, i.e. does not depend on the choice of the cofibrant replacement. This is relevant for our purpose since key lemmas involved in proving this fact are the exact analogues of the results of \cite{Lavau} we are interested in (see proposition \ref{machinerysemimodel}).

\mypara The natural strategy would therefore be to search for a model structure on the category of  Lie $\infty$-algebroids. However, with the definition of $\cite{Lavau}$ (in terms of complexes of {\it vector bundles}), not every singular foliation admits a cofibrant replacement (section 1.1 of $\cite{Lavau}$). We therefore need to allow ourselves to consider what we call {\it $L_\infty$-algebroids} in definition \ref{holiea}, i.e. a complex $L$ of {\it A-modules} equipped with an $L_\infty$-structure compatible with the A-module structure. The nuance between Lie $\infty$-algebroids and  $L_\infty$-algebroids is that for a Lie $\infty$-algebroid, the module $L$ consists in sections of a differential graded vector bundle. 

However, as remarked by Nuiten in \cite[Example 3.1.12]{Nuiten}, even in the more general setting of $L_\infty$-algebroids, an axiom of the definition of model category fails to be satisfied. Therefore, Nuiten has instead equipped the category $\holie$ of $L_\infty$ algebroids with a {\it semi-model category}, a relaxed version of the concept of model category (see definition \ref{semi-model}).

We can still apply this machinery to singular foliations, but in order to do so, we need the analogues of \cite[Lemma 5.1]{DS} and \cite[Lemma 4.9]{DS} for semi-model categories. This is why we start section \ref{tools} by stating the proposition \ref{machinerysemimodel}, whose proof is essentially the same as for model categories. The only "work" consists in checking that the axioms involved in the classical case are still valid for a semi-model category. In order to facilitate this task for the reader, we recall the
outline of the classical proof in the appendix.

We then give in section \ref{semi} details about Nuiten's construction. After recalling in subsection \ref{rapleft} the precise statement and the relevant definitions, we recall the lemma \ref{lemmatransfer} which enables to transport a semi-model structure  via an adjunction. The next step is to describe in the subsection \ref{slice} the semi model structure on $Mod/A$ that can be transported to $\holie$ via the Free/Forget adjunction. We conclude this section by describing thoroughly in subsections \ref{catholie} and \ref{freeLR}  the free functor involved in this adjunction.  

With this at hand we are able to compare in section \ref{cofibrantreplacement} the results of proposition \ref{tools} applied to  $\holie$ with the results of \cite{Lavau} that we wish to recover. We conclude with in subsection \ref{open} a list of open questions.

\paragraph{Aknowledgements} We would like to thank François Métayer for a discussion about filtered colimits and Chris Rogers for mentioning the reference \cite{Nuiten}. A large part of this work was carried out at the Max Planck Institute for Mathematics in Bonn and we are grateful for the excellent resources and working conditions there.

\section{Main technical tool}\label{tools}
%%%%%%%%%%%%%%%%%%%%%%%%%%%%%%%%%%%%%%%%%%%%%%%%%%%%%%%

In \cite{Lavau}, the universal Lie $\infty$-algebroids of a singular foliation \cite[Definition 1.5]{Lavau} is proven to behave in the same manner as a cofibrant replacement in a left semi-model category would with respect to homotopies.

The relevant definitions of left semi-model category, cofibrant replacement, and left and right homotopies can be found in Definition~\ref{semi-model}, Definition~\ref{cofibrantrep}, and \ref{lrhomotopy} respectively. The definition of left semi-model category is taken from \cite[Section 4.4.1]{Fresse} and \cite[Definition 1]{Spitzweck} and \cite[Definition 1.5]{Barwick}.

To be more precise, \cite[Theorem 1.6, 1.8, Proposition 1.22]{Lavau} are reminiscent respectively of Proposition~\ref{machinerysemimodel}.\ref{1}, .\ref{5} and .\ref{3} below:

\begin{prop}\label{machinerysemimodel}
Let $\mathcal{C}$ be a left semi-model category. Suppose that $A$, $X$ and $Y$ are objects in $\mathcal{C}$.

\begin{enumerate}
\item \label{1} Every object in $\mathcal{C}$ has a cofibrant replacement.
\item\label{3}
If $A$ is cofibrant, then $\overset{l}{\sim}$ is an equivalence relation on $Hom_\mathcal{C}(A,X)$. We denote the set of equivalence relations by $\pi^l (A,X)$.
\item \label{2}
If $A$ is cofibrant and $p\colon Y\to X$ is an acyclic fibration, then composition with $p$ induces a bijection:
$$
p_\ast\colon \pi^l (A,Y)\to \pi^l (A,X), \ [f]\mapsto [pf].
$$
\item \label{5} Any two cofibrant replacements of $A$ are weak equivalent and any two such weak equivalences between them are left homotopic.
\end{enumerate}
\end{prop}

We flesh out this analogy in Section \ref{app} of this work by applying this proposition to $\mathcal{C}=L_\infty/\mathcal{O}_M$ (Definition~\ref{holiea}), the category of $L_\infty$-algebroids over a manifold $M$.

\begin{rem}
In what follows we stop writing ``left'' and simply write semi-model category for a left semi-model category.
\end{rem}

\section{A semi-model structure on $\bm{\holie}$}\label{semi}
We remind the reader that throughout this section we consider categories of differential non-negatively graded modules over some base field $\mathbbm{k}$ of characteristic 0 and degree preserving morphisms unless otherwise stated.

In order to be able to apply Proposition~\ref{machinerysemimodel} to the category of $L_\infty$-algebroids, we need to equip it with a cofibrantly generated semi-model structure. In Section~\ref{rapleft} we state Theorem~\ref{transfer} \cite{Nuiten}, which gives a semi-model structure on $\holie$, the category of $L_\infty$-algebroids, from the model structure described in Section~\ref{slice} on $Mod/T_A$, the category of $dg\text{-}A\text{-}$modules over $T_A$. In Section~\ref{catholie} we define the category $\holie$, which is the main focus of this work, and in Section~\ref{freeLR} we describe the free functor $LR\colon Mod/T_A\to \holie$ that allows to apply Lemma~\ref{lemmatransfer}.

\subsection{Statement of the result}\label{rapleft}
In this subsection we present Theorem~\ref{transfer}, which allows us to use the machinery available in model categories for $\holie$ (Definition~\ref{holiea}). This theorem is originally proved for $dg$-algebras and unbounded modules, but for our purposes we just need to state it for an algebra $A$ concentrated in degree 0 and non-negatively differential graded modules.

We start by stating the definition of semi-model category.

%%%%%%%%Def semi model
\begin{df}\label{semi-model}
A {\bf semi-model category} is a bicomplete category $\mathcal{C}$ with three subcategories $\mathcal{W}$, $Fib$ and $Cof$, each containing all identity maps, such that:
\begin{enumerate}
\item[SM1)] $\mathcal{W}$ has the 2-out-of-3 property. As well, $\mathcal{W}$, $Fib$ and $Cofib$ are stable under retracts.
\item[SM2)] \label{SM2} The elements of $Cofib$ have the left lifting property with respect to $\mathcal{W}\cap Fib$. The elements of $\mathcal{W}\cap Cofib$ with cofibrant domain have the left lifting property with respect to fibrations.
\item[SM3)]\label{semifactorization} Every map can be factored functorially into a cofibration followed by a trivial fibration. Every map with cofibrant domain can be factored as a trivial cofibration followed by a fibration.
\item[SM4)] The fibrations and trivial fibrations are stable under transfinite composition, products, and base change.
\end{enumerate}
\end{df}

\begin{rem}\label{semifrommod}
The definition of a semi-model category is a weakening of the notion of {\bf closed model category} which needs in addition the following modifications to the axioms:
\begin{enumerate}
    \item[SM2' \& SM3'] The axioms SM2 and SM3 hold without assumption of cofibrancy of the domain.
\end{enumerate}
One does not require for a closed model category the axiom SM4 which follows from the other axioms.
\end{rem}

\begin{rem}
In particular, any model category is a semi-model category.
\end{rem}
%%%%%%%%%%%%%%%%%%%%%%%%%%%%%%%%%%%%%%%%%%%%%%%%%%%%%%%%%%%%%%%%%%%%%%%%%%%%%%%%%%%%%%%%%%%%%%%%%%%%%%

\begin{thm}\cite[Lemma 2.14, Theorem 3.1]{Nuiten}\label{transfer}
Let $A$ be an algebra. Then the category $\holie$ of $L_\infty$-algebroids over $A$ admits a semi-model structure, in which a map is a weak equivalence (resp. a fibration) if and only if it is a quasi-isomorphism (surjection on positive degrees).
\end{thm}

In fact, this structure is cofibrantly generated \cite[Definition 11.1.2]{HirschModel}, \cite[Chapter 11]{HirschModel} and \cite[Section 2.1]{Hovey}, which means that there is a set of maps generating the class of trivial fibrations through the right lifting property.
%%%%%%%%%%%%%%%%%%%%%%%%%%%%%%%%%%%%%%%%%%%%%%%

Theorem~\ref{transfer} is proven through the following version of the transfer lemma:
%%%%%%%%%%%%%%%%%%%%%%%%%%%%%%%%%%%%%%%%%%%%%%%%%%
%%%%%%%%%%%%%%%%%Transfer lemma%%%%%%%%%%%%%%%%%%%%
\begin{lem}[Transfer of semi-model structure]\label{lemmatransfer}
Let $F\colon \mathcal{M}\leftrightarrow \mathcal{N}\cocolon G$ be an adjunction between locally presentable categories and suppose that $\mathcal{M}$ carries a semi-model structure with sets of generating (trivial) cofibrations $I$ and $J$.

Define a map in $\mathcal{N}$ to be a weak equivalence (fibration) if its image under $G$ is a weak equivalence (fibration) in $\mathcal{M}$ and a cofibration if it has the left lifting property against the trivial fibrations. Assume that the following condition holds:
\begin{itemize}
    \item Let $f\colon A\to B$ be a map in $\mathcal{N}$ with cofibrant domain, obtained as a transfinite composition of pushouts f maps in $F(J)$. Then $f$ is a weak equivalence.
\end{itemize}
Then the above classes of maps determine a tractable semi-model structure on $\mathcal{N}$ whose generating (trivial) cofibrations are given by $F(I)$ and $F(J)$.
\end{lem}
%%%%%%%%%%%%%%%%%%%%%%%%%%%%%%%%%%%%%%%%%%%%%%%%%%

The transfer takes the model structure on the category $Mod/T_A$ (Definition~\ref{anchored}) and gives a semi-model structure on $\holie$ through an adjunction
\begin{center}
\begin{tikzcd}
\holie
\arrow[bend left=35]{r}[name=F]{F}
& Mod/T_A
\arrow[bend left=35]{l}[name=LR]{LR}
\end{tikzcd}
\end{center}
The existence of such adjunction is the content of Proposition~\ref{freelr} below.

%%%%%%%%%%%%%%%%%%%%%%%%%%%%%%%%%%%%%%%%%%%%%%%%%%%%%
\begin{proof}[Proof of Theorem~\ref{transfer}]
The constructions in \cite{Nuiten} used to prove the unbounded case preserve the subcategory $dg\text{-}A\text{-}Mod^{\geq0}$. Thus, the result follows in this case.
\end{proof}

%%%%%%%%%%%%%%%%%%%%%%%%%%%%%%%%%%%%%%%%%%%%%%%%%%%%%%%%%%%%%%%%%%%%

\subsection{The model structure on $\bm{Mod/A}$}\label{slice}
The objective of this subsection is to prove Proposition~\ref{modinMODA}, where we give a model structure on the slice category $Mod/T_A$, the category of anchored differential non-negatively graded modules over an algebra, which we describe as a slice category in Definition~\ref{anchored}. We then state a general result, Theorem~\ref{modelunder}. It enables to obtain the model structure on $Mod/T_A$ from the model structure on $dg\text{-}A\text{-}Mod^{\geq0}$ given in Theorem~\ref{modinMOD}.

%%%%%%%%%%%%%%%%%%%%%%%Slice category%%%%%%%%%%%%%%%%%%%%%%%%%%%%%

Recall the classical notion of slice category :

\begin{df} Let $\mathcal{M}$ be a category and $Z$ and object of $\mathcal{M}$. The {\bf slice category} $\mathcal{M}/Z$ has for objects arrows of $\mathcal{M}$ the form $x\rightarrow z$
and for arrows commutative diagrams in $\mathcal{M}$ of the form
\begin{center}
\begin{tikzpicture}[normal line/.style={->},font=\scriptsize]
\matrix (m) [matrix of math nodes, row sep=2em,
column sep=1.5em, text height=1.5ex, text depth=0.25ex]
{ X& & Y \\
 & Z. & \\ };

\path[normal line]
(m-1-1) edge  (m-1-3)
edge  (m-2-2)
(m-1-3) edge  (m-2-2);
\end{tikzpicture}
\end{center}
\end{df}

In what follows, we denote by $T_A$ the $A$-module of derivations of $A$. We can look at $T_A$ as a graded $A$-module and if $A$ is concentrated in degree 0, so is $T_A$.
\begin{df} \label{anchored}
We denote by $\bm{Mod/T_A}$ the slice category $dg\text{-}A\text{-}Mod^{\geq0}/T_A$. We refer to such an object $\rho\colon M\to T_A$ as an {\bf anchored module}.
\end{df}

\begin{prop}\label{modinMODA}
The category $Mod/T_A$ has a cofibrantly generated semi-model structure.
\end{prop}

In Remark~\ref{semifrommod} we highlight the fact that a model structure is a particular example of a semi-model structure. So, we actually exhibit a stronger, i.e. a {\it cofibrantly generated} model structure on $Mod/T_A$.%This becomes useful for applying Proposition~\ref{transfer}.

\begin{thm}\cite[Theorem 7]{Hirsch}
  \label{modelunder}
  Let $\mathcal{M}$ be a cofibrantly generated model category (see
  \cite[Definition~11.1.2]{HirschModel}) with generating cofibrations $I$ and
  generating trivial cofibration $J$, and let $Z$ be an object of
  $\mathcal{M}$.  If
  \begin{enumerate}
  \item $I_{Z}$ is the set of maps in $\mathcal{M}/{Z}$ of the form
   \begin{center}
  \begin{equation} \label{morphabove}
\begin{tikzpicture}[normal line/.style={->},font=\scriptsize]
\matrix (m) [matrix of math nodes, row sep=2em,
column sep=1.5em, text height=1.5ex, text depth=0.25ex]
{ X& & Y \\
 & Z & \\ };

\path[normal line]
(m-1-1) edge  (m-1-3)
edge  (m-2-2)
(m-1-3) edge  (m-2-2);
\end{tikzpicture}
\end{equation}
\end{center}
    in which the map $X \to Y$ is an element of $I$ and
  \item $J_{Z}$ is the set of maps in $\mathcal{M}/{Z}$ of the
    form \eqref{morphabove} in which the map $X \to Y$ is an element of
    $J$,
  \end{enumerate}
  then the standard model category structure on $\mathcal{M}/{Z}$
  (in which a map \eqref{morphabove} is a cofibration, fibration, or
  weak equivalence in $\mathcal{M}/{Z}$ if and only if the map $X
  \to Y$ is, respectively, a cofibration, fibration, or weak
  equivalence in $\mathcal{M}$) is cofibrantly generated, with generating
  cofibrations $I_{Z}$ and generating trivial cofibrations $J_{Z}$.
\end{thm}

%%%%%%%%%%%%%%%%%%%%%%%Cofibrantly generated%%%%%%%%%%%%%%%%%%%%%%%%%%%%%%%
We introduce the classes of cofibrations and trivial fibrations in $dg\text{-}A\text{-}Mod^{\geq0}$:

Let the ``disc" $D_A(n)$ denote the chain complex $A[n]\xrightarrow{Id_A} A[n-1]$, and the ``sphere"  $S_A(n)$ denote the chain complex $A[n]$ with trivial differential. Let $J$ be the set of morphisms of complexes of the form $0\to D_A(n)$ and let $I$ be the set of inclusions $S_A(n-1)\to D_A(n)$.

\begin{thm}\cite[Theorem 2.3.11]{Hovey}\label{modinMOD}
$dg\text{-}A\text{-}Mod^{\geq0}$ is a cofibrantly generated model category with $I$ as its generating set of cofibrations, $J$ as its generating set of trivial cofibrations, and quasi-isomorphisms as its weak equivalences. The fibrations are the surjections.
\end{thm}

\begin{proof}[Proof of Proposition~\ref{modinMODA}]
We apply Theorem \ref{modelunder} to the slice category $Mod/T_A$ and the cofibrantly generated model structure on $dg\text{-}A\text{-}Mod^{\geq0}$ given by Theorem~\ref{modinMOD}.
\end{proof}

\subsection{The category $\bm{\holie}$}\label{catholie}

In this subsection we define the category of $L_\infty$-algebroids associated to a graded algebra $A$. In broad terms, such an algebroid is given by an $L_\infty$-algebra with a compatible $A$ module structure in the same spirit as a Lie algebroid in \cite{Kapranov}. In the following we consider non-negatively graded dg-modules.

\begin{df}\label{holiea}
\begin{enumerate}
\item An {\bf $\bm{L_\infty}$-algebra} is a module $L$ with an $L_\infty$-structure \cite[Definition 2]{Vitagliano2} given by degree $k-2$ anti-symmetric multibrackets
$$
[\ ]_k\colon \Lambda^k (L)\to L,
$$
with
$$
\sum_{i+j=k}(-1)^{ij}\sum_{\sigma\in S_{i,j}}\chi(\sigma, v)\left[[v_{\sigma(1)},\ldots,v_{\sigma(i)}]_i,v_{\sigma(n+1)},\ldots,v_{\sigma(i+j)}\right]_{j+1}=0.
$$
Here, $\chi(\sigma, v)$ is the sign for which $v_1\wedge\cdots \wedge v_k=\chi(\sigma, v)v_{\sigma(1)}\wedge\cdots\wedge v_{\sigma(k)}$.

We remind the reader that $\bigwedge^k L$ is the quotient of the graded vector space $\bigotimes^k L$ under the equivalence relation linearly generated by
$$
v_1\otimes\cdots\otimes v_k=-(-1)^{\lvert v_i\rvert\lvert v_{i+1}\rvert}v_1\otimes\cdots\otimes v_{i+1}\otimes v_i\otimes\cdots\otimes v_k
$$
for all $i$. As well, $S_{i,j}$ is the group of $(i,j)$-unshuffles, which is to say, the set of permutations $\sigma$ of $1,\ldots, k=i+j$ such that
$$
\sigma(1)<\cdots<\sigma(i) \text{ and }\sigma(i+1)<\cdots<\sigma(k).
$$

\item An $L_\infty$-algebra $L$ with an $A$-module structure (an {\bf $\bm{A\text{-}L_\infty}$-algebra}) is an {\bf $\bm{L_\infty}$-algebroid} if it has a compatible action on $A$. The action is specified by an $A$-module morphism (which we call the anchor)
$$
\rho\colon L\to T_A
$$
of degree $0$. The compatibility conditions for the anchor and brackets are given by:
\begin{gather*}
[v_0,av_1]_2=(-1)^{\lvert a\rvert\lvert v_0\rvert}a[v_0,v_1]_2+\rho(v_0)[a]v_1\\
\left[v_0,\ldots,av_{k}\right]_{k+1}=(-1)^{\lvert a\rvert \left(k+\lvert v_0\rvert+\cdots+\lvert v_{k-1}\rvert\right)}a\left[v_0,\ldots,v_{k}\right]_{k+1}
\end{gather*}
for $a$ in $A$ and $v_i\in L$.

We denote by $\bm{\holie}$ the category of $L_\infty$-algebroids with morphisms given by the $A$-module morphisms that preserve the multibrackets.
\end{enumerate}
\end{df}

\begin{rem}\label{remarkLR}
Since the higher brackets of $T_A$ all vanish, the data of an $L_\infty$-algebroid on an $A$-module $L$ is equivalent to that of an $L_\infty$-algebra structure on $L$ where the higher brackets are $A$-linear on the first coordinate, together with a strict morphism \cite[Definition 5]{Vitagliano2} of $L_\infty$-algebras to $T_A$.

An $L_\infty$-algebroid is a version up to homotopy of the Lie-Rinehart algebra defined in \cite[Section 1]{Kapranov}. The higher brackets and the jacobiators give the structure up to homotopy.
\end{rem}

\subsection{The free functor $\bm{LR}$}\label{freeLR}
In Proposition~\ref{free linf1} we describe a right adjoint for the forgetful functor
$$
F\colon \holie\to Mod/T_A,
$$
which takes an $L_\infty$-algebroid and gives back its underlying module. The existence of this functor is implicit throughout \cite[Section 3]{Nuiten}, but here we give an explicit construction.

We start by presenting the $L_\infty$-words on a fixed $M$ (not necessarily bounded) $A$-module ($A$ possibly differential graded): An {\bf $\bm{L_\infty^A}$-word} of arity $k$ is recursively defined as an element of
$$
A\otimes\left(\bigwedge^k M'\right)[k-2],
$$
where $M'$ is an $A$-module of some previously defined $L_\infty^A$-words. We represent one such word by a symbol
$$
a\left[v_0,\ldots,v_{k-1}\right]_k,
$$
where $a$ is an element in $A$, while the $v_i$ are all previously constructed $L_\infty^A$-words. The degree of $a\left[v_0,\ldots,v_{k-1}\right]_k$ is defined to be
$$
deg\left(a\left[v_0,\ldots,v_{k-1}\right]_k\right)=\sum_i deg(v_i)+ \lvert a\rvert +k-2.
$$

For non-negatively graded $M$ and non graded $A$, it is necessary to modify this definition. The recursive construction goes as follows: We start setting all the elements of $\lambda_0=M$ as $L_\infty^A$-words.

We can then form the symbols that involve elements of $\lambda_0$. The set of all such symbols forms an $A$-module which can be expressed as
$$
\lambda'_1\coloneqq \bigoplus_{k\geq 1} A\otimes \left(\bigwedge\nolimits^k \lambda_0\right)[k-2]/K_0.
$$
Here we take $\bigwedge^1 \lambda_{0}[1]$ to be $\lambda_{0}[1]$ and $K_0$ to be the $A$-submodule of negatively graded elements. Thus, the resulting module is non-negatively graded. Including the elements of $\lambda_0$ themselves we obtain
$$
\lambda_1\coloneqq \lambda_0\oplus \lambda'_1.
$$

Recursively, once defined $\lambda_n$, the $A$-module of all $L_\infty^A$-words with at most $n$ brackets, we define
$$
\lambda'_{n+1}\coloneqq \bigoplus_{k\geq 1} A\otimes \left(\bigwedge\nolimits^k \lambda_{n}\right)[k-2]/K_n,
$$
where again, $K_n$ is the $A$-submodule of negatively graded elements, and then
$$
\lambda_{n+1}\coloneqq \lambda_0\oplus \lambda'_{n+1}.
$$
It is clear from the construction that $\lambda_n\subset \lambda_{n+1}$ for all $n$.

\begin{df}
Given an $A$-module $M$, we define the {\bf set of $\bm{L_\infty^A}$-words on $\bm{M}$} as the colimit of $A$-modules
$$
L(M)=colim\ \lambda_{i}.
$$
\end{df}

Given an element $v\in \lambda_n$, we say that $w\in \lambda_{n+1}$ is {\bf directly over} $v$ if there are $v_1,\ldots,v_{k-1}\in \lambda_n$ and $a$ in $A$ such that $w=[a v,v_1,\ldots,v_{k-1}]_k$. We say that {\bf$\bm{w}$ is over $\bm{v}$} if there is a sequence $v=v_0,\ldots,v_k=w$ such that $v_{i+1}$ is directly over $v_i$.

\begin{df}
The {\bf free $\bm{L_\infty^A}$-algebra} associated to $M$ is the $A$-module given by
$$
L_\infty^A(M)\coloneqq L(M)/\sim,
$$
which is the quotient of the set of $L_\infty^A$-words of $M$ under $\sim$, which is the $A$-module generated by all the elements over the jacobiators
$$
\sum_{i+j=k}(-1)^{ij}\sum_{\sigma\in S_{i,j}}\chi(\sigma, v)\left[[v_{\sigma(1)},\ldots,v_{\sigma(i)}]_i,v_{\sigma(n+1)},\ldots,v_{\sigma(i+j)}\right]_{j+1}.
$$
with $\chi(\sigma,v)$ the sign for which $v_1\wedge\cdots \wedge v_k=\chi(\sigma, v)v_{\sigma(1)}\wedge\cdots\wedge v_{\sigma(k)}$.

The $L_\infty$-brackets on $L_\infty^A(M)$ are defined in the obvious way: given $k+1$ elements of $L_\infty^A(M)$ represented by $L_\infty^A$ words $v_0,\ldots,v_k$, their bracket is given by the class represented by $[v_0,\ldots, v_k]_{k+1}$.

These brackets are well defined on $L_\infty^A(M)$ and their jacobiators vanish from construction. 
\end{df}

In the construction of $L_\infty^A(M)$ we did not make use of the internal differential structure nor the anchors of the module $M$. These are considered below in the construction of the free $L_\infty$-algebroid associated to $M$.

\begin{prop}\label{free linf1}
Given a graded $A$-Module $M$, the construction $L_\infty^A(M)$ is functorial and is a left adjoint to the forgetful functor from the category of $A$-modules with an $L_\infty$-algebra structure and strict morphisms of $A$-modules,  to the category of $A$-Modules.
\end{prop}

\begin{proof}
Given a morphism of $A$-modules $f\colon M\to N$, we give a morphism of $A$-modules $L_\infty^A(f)\colon L_\infty^A(M)\to L_\infty^A(N)$ by defining it on the level of $L_\infty^A$-words: we start by setting $L(f)m\coloneqq f(m)$ for $m$ in $M$, and recursively define
\begin{equation}\label{equation}
L(f)a\left[v_0,\ldots, v_{k-1}\right]_k\coloneqq a\left[L(f)v_0,\ldots, L(f)v_{k-1}\right]_k.
\end{equation}
This is clearly a morphisms of $A$-modules that preserves jacobiators. Therefore, it descends to the required morphism $L_\infty^A(f)\colon L_\infty^A(M)\to L_\infty^A(N)$. Since it preserves the $L_\infty$-brackets by construction, it is a strict $L_\infty$-morphism.

The adjunction property can be verified using that $M$ is an $A$-submodule of $L_\infty^A(M)$. Therefore, given any $A$-module, bracket preserving morphism ({\bf$\bm{A\text{-}L_\infty}$-morphism}) $F\colon L_\infty^A(M)\to L$, the restriction to $M$ is an $A$-module morphism.

As well, given an $A\text{-}L_\infty$-algebra $L$, Equation~\ref{equation} above defines the unique strict $L_\infty$-extension to $L_\infty^A(M)$ of a given $A$-module morphism $f\colon M\to L$. The verification that these two assignations follows from the construction of the $L_\infty$-extensions.
\end{proof}

\begin{rem}
Proposition~\ref{free linf1} shows that $L_\infty^A(M)$ is a free object associated to $M$ in $A\text{-}L_\infty$-algebras.
\end{rem}

\begin{rem}\label{anchors}
If $\alpha\colon M\to T_A$ is an anchored module, so is $L_\infty^A(M)$. In fact, we obtain an $A\text{-}L_\infty$-morphism $\tilde{\alpha}\colon L_\infty^A(M)\to T_A$ from $\alpha$ through the described adjunction
\begin{center}
\begin{tikzcd}
A\text{-}L_\infty \arrow[bend left=35]{r}[name=F]{F} & Mod \arrow[bend left=35]{l}[name=L]{L}
\end{tikzcd}
\end{center}
Here we consider the dg-Lie algebra $T_A$ as an $A$-module with the $L_\infty$-structure coming from the Lie bracket.
\end{rem}

\begin{prop}\label{freelr}
Let $\alpha\colon M\to T_A$ be in $Mod/A$ with differential $d$. The $A$-module over $T_A$ given by the quotient
$$
LR(M)\coloneqq L_\infty^A(M)/\sim
$$
is an $L_\infty$-algebroid. The anchor $\rho\colon LR(M)\to T_A$ is the quotient function of the map $\hat{\alpha}\colon L_\infty^A(M)\to T_A$ described in Remark~\ref{anchors}.

Here, $\sim$ is the $A\text{-}L_\infty$-ideal generated by the relations
\begin{gather*}
[v]_1=d(v)\\
[v_0,av_1]_2=
a[v_0,v_1]_2+\hat{\alpha}(v_0)(a)v_1\\
\left[v_0,\ldots,av_{k}\right]_{k+1}=
a\left[v_0,\ldots,v_{k}\right]_{k+1}
\end{gather*}
for $v$ in $M$, the $v_i$ in $L_\infty^A(M)$, and $a$ in $A$.

Finally, this construction is functorial and gives a right adjoint for the forgetful functor $F\colon \holie\to Mod/A$.
\end{prop}

\begin{rem}
An {\bf$\bm{A\text{-}L_\infty}$-ideal} in an $A\text{-}L_\infty$-algebra $L$ generated by a set $S$ is the smallest $A$-submodule of $L$ that is an ideal for the $L_\infty$-brackets.
\end{rem}

\begin{proof}
From the definition of the quotient, the $L_\infty$-brackets of $L_\infty^A(M)$ descend to $L_\infty$-brackets on $LR(M)$ and, since $T_A$ is already an $L_\infty$-algebroid, so does the anchor.

The compatibility conditions for the brackets and anchor from Definition~\ref{holiea} are exactly the quotient relations and so, $LR(M)$ is itself an $L_\infty$-algebroid.

Given $\alpha_M\colon M\to T_A$ and $\alpha_N\colon N\to T_A$ and $f\colon M\to N$ in $Mod/A$, from Proposition~\ref{free linf1} we obtain a commutative diagram of $A\text{-}L_\infty$-algebras
\begin{center}
    \begin{tikzcd}[column sep=small]
    L_\infty^A(M)
    \arrow[rr, "L(f)"]
    \arrow[rd, swap, "\hat{\alpha}_{M}"]
    && L_\infty^A(N)
    \arrow[ld, "\hat{\alpha}_{N}"]
    \\
    & T_A
    \end{tikzcd}
\end{center}
This induces a morphim $LR(f)\colon LR(M)\to LR(N)$ in $\holie$.

The proof that the functors $F\colon \holie\leftrightarrow Mod/A\cocolon LR$ are adjoint to each other is similar to that in the proof of Proposition~\ref{free linf1}.
\end{proof}

\begin{df}\label{freehlr}
We define the {\bf free $\bm{L_\infty}$-algebroid} associated to $\alpha\colon M\to T_A$ in $Mod/A$ to be the $L_\infty$-algebroid $LR(M)$.
\end{df}

%%%%%%%%%%%%%%%%%%%%%%%%%%%%%%%%%%%%%%%%%%%%%%%%%%%%%%%%

\section{Applications to singular foliations}\label{app}
In this section we propose results  parallel to those of \cite[Section 1]{Lavau} for Lie $\infty$-algebroids in terms of the (semi-)model category structure on $\holie$ given by Theorem~\ref{transfer}. We first show how singular foliations and Lie $\infty$-algebroids \cite{Lavau} are examples of $L_\infty$ algebroids. With this connection in mind, we then compare each statement of \cite{Lavau} with its semi-model category analogue that we introduce and prove as corollaries of our main technical lemma \ref{machinerysemimodel}. We conclude with open questions raised by the discrepancies between the two sets of results.

\subsection{Universal Lie $\bm{\infty}$-algebroids}
In this subsection, in order to make the connection with the semi-model structure on $L_\infty/\mathcal{O}_M$,  we show that both Lie $\infty$-algebroids and singular foliations can be seen as objects in $L_\infty/\mathcal{O}_M$. We fix a smooth manifold $M$ and its ring of smooth functions $\mathcal{O}_M$. Recall

\begin{df}\cite[Definition 1.13]{Lavau}
A {\bf Lie $\infty$-algebroid} is a non-positively graded differential vector bundle over a manifold $M$ whose space of sections has a strong homotopy Lie-Rinehart algebra structure. This is to say, an $LR_\infty[1]$-algebra in the sense of \cite[Definition 7]{Vitagliano2} over the algebra $\mathcal{O}_M$ of smooth functions on $M$.
\end{df}

\begin{rem}
In \cite{Lavau}, the definitions are given in terms of non-positive cohomological chain complexes. We can directly transport our results to this setting by using the isomorphim of this category to that of non-negative homological chain complexes.
\end{rem}

Then one has

\begin{lem}\label{inclusion}
The de-suspension of the space of sections of an $\infty$-algebroid forms an $L_\infty$-algebroid.
\end{lem}

\begin{proof}
 The compatibility conditions of the anchor and brackets in \cite[Definition 1.13]{Lavau} are the higher Jacobi equations for the $LR_\infty[1]$-algebra. In this case, since both $\mathcal{O}_M$ and $T_{\mathcal{O}_M}$ are concentrated in degree 0, the action of the algebroid on $\mathcal{O}_M$ has no higher terms. Therefore, taking the $L_\infty$-algebroid associated to this $LR_\infty[1]$-algebra we obtain (after desuspension) an object in $L_\infty/\mathcal{O}_M$.
 \end{proof}

We can also express the notion of  singular foliation in terms of $L_\infty$-algebroid. First recall 

\begin{df}
A {\bf singular foliation} over $M$ is a locally finitely generated $\mathcal{O}_M$-submodule of $\mathfrak{X}(M)$ that is closed under the Lie bracket.
\end{df}

 Since $T_{\mathcal{O}_M}=\mathfrak{X}(M)$ (with the $L_\infty/{\mathcal{O}_M}$-structure given by the commutator of derivations and anchor the identity) is just $\mathfrak{X}(M)$, one can restate the definition in terms of the $L_\infty$-structure.

\begin{lem}
A singular foliation $\mathcal{F}$ over $M$ is a locally finitely generated $\mathcal{O}_M$-module $\mathcal{F}\subset\Tsmooth$ that is closed under the $L_\infty/{\mathcal{O}_M}$-structure of $\Tsmooth$.
\end{lem}

\subsection{Cofibrant replacements in $\bm{L_\infty/{\mathcal{O}_M}}$}\label{cofibrantreplacement}

We can now show that results similar to those of \cite{Lavau} can be understood in terms of the model structure on $Mod/T_{\mathcal{O}_M}$ (see Proposition~\ref{modinMODA}) and the semi-model structure on $L_\infty/\mathcal{O}_M$ (see Theorem~\ref{transfer}).

We first recall the definition of a resolution of a singular foliation \cite[Definition 1.1]{Lavau}:

\begin{df}
A {\bf resolution of a singular foliation $\mathcal{F}$} is a complex of vector bundles for which the cohomology of the associated complex of sections gives the  $\mathcal{O}_M$-module $\mathcal{F}$.
\end{df}

The model structure on $Mod/T_{\mathcal{O}_M}$ given by Proposition~\ref{modinMODA} is such that:

\begin{lem}
A resolution of a singular foliation $\mathcal{F}$ is a cofibrant replacement in $Mod/T_{\mathcal{O}_M}$.
\end{lem}

We remind the reader of the definition of cofibrant replacements in a (semi-) model structure.

\begin{df}\label{cofibrantrep}
A cofibrant replacement of an object $X$ in a (semi-) model category $\mathcal{C}$ with initial object $0$ is an object $QX$ such that there is a factorization of the initial arrow $0\to X$ as $0\xrightarrow{i}QX\xrightarrow{p} X$, with $i$ a cofibration and $p$ a trivial fibration.
\end{df}

The following remark is the main reason for us to work with $L_\infty$-algebroids instead of Lie-$\infty$ algebroids:
\begin{rem}
Resolutions of singular foliations do not always exist (see section 1.1 of $\cite{Lavau}$), while the existence of a cofibrant replacement in $Mod/\mathcal{O}_M$ of a singular foliation is guaranted by the axiom $SM3$ of a semi-model category. 
\end{rem}

If a resolution of $\mathcal{F}$ exists, on can consider the question of equipping it with a Lie-$\infty$ algebroid structure.

\begin{df}\cite[Definition 1.5]{Lavau}\label{universal}
A resolution of $\mathcal{F}$ is called a {\bf universal Lie-$\infty$ algebroid over $\mathcal{F}$} if the differential on the sections can be extended to a homotopy Lie-Rinehart structure.
\end{df}

One of the main results of \cite{Lavau} (Theorem 1.6) reads: 

\begin{thm}
If $\mathcal{F}$ admits a resolution, then this resolution can be upgraded to a universal Lie-$\infty$ algebroid over $\mathcal{F}$.
\end{thm}

 The analogous notion in terms of the semi-model category on $L_\infty/\mathcal{O}_M$ is given by a cofibrant replacement. The existence of such cofibrant replacement is then a direct consequence of the axioms:

\begin{thm}\label{replacement}
For any singular foliation $\mathcal{F}$ there exists a replacement by a dg-$\mathcal{O}_M$-module endowed with a structure of $L_\infty$-algebroid over $\mathcal{F}$.
\end{thm}
\begin{proof}
It suffices to apply Proposition~\ref{machinerysemimodel}.\ref{1} to the semi-model structure on $L_\infty/\mathcal{O}_M$ given by Theorem~\ref{transfer}.
\end{proof}

The terminology "universal" in definition \ref{universal} is due to the following result (\cite[Theorem 1.8]{Lavau}):

\begin{thm}\label{liftinglavau}
Two universal Lie $\infty$-algebroids over $\mathcal{F}$ are quasi-isomorphic in the category of Lie $\infty$-algebroids, and such a quasi-isomorphism is essentially unique (up to homotopy).
\end{thm}

This result has a direct analog in terms of the semi-model structure on $L_\infty/\mathcal{O}_M$.

\begin{thm}\label{lifting}
Let $\mathcal{F}$ be a singular foliation. Then, any two cofibrant replacements of $\mathcal{F}$  in $L_\infty/\Tsmooth$ are quasi-isomorphic (with strict morphisms of $L_\infty$-algebroids) and any two such isomorphisms are (left)-homotopic.
\end{thm}

\begin{proof}
This is proposition \ref{machinerysemimodel}.\ref{5}.
\end{proof}

Even though Lemma~\ref{inclusion} states that any universal Lie $\infty$-algebroid over $\mathcal{F}$ is a replacement of $\mathcal{F}$ in $L_\infty/\mathcal{O}_M$ (i.e. an element of $L_\infty/\mathcal{O}_M$ weakly equivalent to $\mathcal{F}$), it is a priori \underline{not} a cofibrant replacement. In particular, we can not à priori deduce Theorem~\ref{liftinglavau} from Theorem~\ref{lifting}. Moreover, Theorem \ref{liftinglavau} and \ref{lifting} differ in two other points: 

\begin{rem}\label{homodif}
\begin{itemize}
    \item The morphisms in proposition \ref{liftinglavau} are $L_\infty$-morphisms, while the morphisms in \ref{lifting} are strict morphisms of $L_\infty$-algebras.
\item The notion of homotopy equivalence used in \ref{liftinglavau} is very different from the one we are using which is based on Cylinder objects (see definition \ref{lrhomotopy}.1).
\end{itemize}
\end{rem}

The semi-model category statement analogous to \cite[Proposition 1.22]{Lavau} is the following theorem, which is a corollary of Proposition~\ref{machinerysemimodel}.\ref{5}.

\begin{thm}
Left homotopy of Lie $\infty$-algebroid morphisms is an equivalence relation denoted by $\overset{l}{\sim}$, which is compatible with composition.
\end{thm}

Once again, the notion of homotopy equivalence used in the two statements do not coincide. The reason is that when one works in the realm of model categories, there are standard definitions of left and right homotopies between morphisms, namely Definition~\ref{lrhomotopy}.\ref{leftHomotopy} and \ref{lrhomotopy}.\ref{rightHomotopy} below.

\begin{df}\label{lrhomotopy}
Let $A$ and $X$ be objects in a semi-model category $\mathcal{C}$.
\begin{enumerate}
    \item A {\bf cylinder object} for $A$ is an object $Cyl(A)$ together with a factorization
    \begin{center}
    \begin{tikzcd}
    A\amalg A
    \arrow[rr, bend right,  "id_{A}+ id_{A}"']
    \arrow[r, "i_0+i_1"]
    &Cyl(A)
    \arrow[r, "\sim"]
    &A
    \end{tikzcd}
    \end{center}
    where the map $Cyl(A)\to A$ is a weak equivalence.
    \item Similarly, a {\bf path object} for $X$ is an object $Path(X)$ together with a factorization
    \begin{center}
    \begin{tikzcd}
    X
    \arrow[rr, bend right,  "id_{A}\times id_{A}"']
    \arrow[r, "\sim"]
    &Path(X)
    \arrow[r, "p_0\times p_1"]
    &X\times X
    \end{tikzcd}
    \end{center}
    where the morphism $X\to Path(X)$ is a weak equivalence.
    \item\label{leftHomotopy}Two morphisms $f,g\colon A\to X$ are {\bf left homotopic}, and we write {\bf $f\overset{l}{\sim}g$}, if there is a cylinder object of $A$ and a map $H\colon Cyl(A)\to X$ such that $f+ g\colon A\amalg A\to X$ factorizes as
    $$
    A\amalg A\to Cyl(A)\xrightarrow{H}X
    $$
    \item\label{rightHomotopy} Analogously, $f,g\colon A\to X$ are {\bf right homotopic}, and we write {\bf $f\overset{r}{\sim}g$}, if there is a path object of $X$ and a map $H\colon A\to Path(X)$ such that $f\times g\colon A\to X\times X$ factorizes as
    $$
    A\xrightarrow{H} Path(X)\to X\times X
    $$
\end{enumerate}
\end{df}

\begin{rem}\label{remho}
 Because of the restrictions imposed by Axiom~\emph{SM3}, it is natural to use cylinder objects for the notion of homotopy in $L_\infty/\mathcal{O}_M$, i.e. to use left-homotopies. However, it is  still possible to define right homotopies by a construction of path objects akin to the one found in \cite[Section 2]{Vezzosi} which turns out to be a notion similar to the one used by \cite{Lavau} (see \cite[Remark 6]{Lavau}). But then, it is not clear to us how to prove a statement similar to proposition \ref{machinerysemimodel}.4 but for right homotopies. 
\end{rem}

\subsection{Open questions}\label{open}

The ineffable original goal of this paper was to obtain the results of \cite{Lavau} as genuine properties of a (semi)-model category on $L_\infty/\Tsmooth$. However, Remark~\ref{homodif} implies that one needs to change the (semi)-model structure we consider if one wants to reach this goal. More precisely, this work triggers the following questions.

\begin{que}\label{q1}
Does there exist a variant of a notion of model-category on $L_\infty/\Tsmooth$ for which the morphisms are not anymore necessarily strict?
\end{que}

If such a structure exists we can consider the following questions:

\begin{que}\label{q2}
Is any universal Lie $\infty$-algebroid cofibrant? In other words, is the first part of Theorem~\ref{liftinglavau} a corollary of Theorem~\ref{lifting} (with non necessarily strict morphisms of $L_\infty$-algebroids)?
\end{que}

\begin{que}\label{q3}(Inspired by theorem \ref{remho})
Does the notion of left homotopy coincide with the notion of homotopy considered in \cite{Lavau}? In other words, is the second part of Theorem~\ref{liftinglavau} a corollary of Theorem~\ref{lifting}?
\end{que}

Question \ref{q3} can be divided into two.

\begin{que}
Does the equivalence relation given by left homotopy coincide with the one given by right homotopy? 
\end{que}

\begin{que}\label{q4}
Does the notion of right homotopy coincide with the definition of homotopy used in \cite{Lavau}? 
\end{que}

Another possible generalisation of a $L_\infty$-algebroid is given by allowing the existence of higher anchors which form a homotopy morphism to $T_A$.

\begin{que}\label{q5}
Is there a (semi-)model structure in the category of $L_\infty$-algebroids with homotopy anchors? How does it relate to the other model structures above?
\end{que}

We have already noted (Remark~\ref{remarkLR}) that $L_\infty$-algebroids are particular examples of Lie-Rinehart algebras up to homotopy. On the other hand, in \cite[Section 5]{Kji}, Kjeseth studies the notion of a resolution of a Lie-Rinehart Pair.

 \begin{que}
 Does there exist a variant of a notion of model-category on the category of homotopy Lie-Rinehart pairs for which the resolutions considered in \cite[Definition 5.1]{Kji} are cofibrant replacements?
  \end{que}

\begin{que}
Does it induce, when restricted to a fixed associative algebra A, the structure considered aimed at in question  \ref{q1}?
\end{que}

\appendix
\section{Proof of the main technical tool}
We give a proof of the properties of semi-model categories mentioned in Section~\ref{tools}. The proof follows the lines of the same properties for model categories in \cite[Sections 4 and 5]{DS}.

%%%%%%%%%%%%%%%%%%%%%%%%%%%%%%%%%%%%%%%%%%%%%%%%%%%
\begin{proof}[Proof of Proposition~\ref{machinerysemimodel}]
For this proof we follow the proof of \cite[Propositon 1.2.5]{Hovey}.
\begin{enumerate}
    \item From Axiom~\emph{SM3} of Definition~\ref{semi-model}, the maps from the initial object $0\to A$ can be factorized as a cofibration followed by a trivial fibration
    $$
    0\to QA\to A.
    $$
    
    Then, $QA$ is the cofibrant replacement.
    \item Reflexivity for a map $f$ is given by the homotopy
    $$
    A\amalg A\to Cyl(A)\xrightarrow{\sim} A\xrightarrow{f} X,
    $$
    for any cylinder object.
    
    Symmetry is given by reversing the order of homotopies. That is, given a homotopy $Cyl(A)\xrightarrow{H} X$, we give a new homotopy as
    $$
    A\amalg A\xrightarrow{s}A\amalg A\xrightarrow{i_0+i_1}Cyl(A)\xrightarrow{H} X,
    $$
    where $s\colon A\amalg A\to A\amalg A$ is the morphism that switches components.
    
    For transitivity consider maps $f$, $g$ and $h$ such that $f\overset{l}{\sim} g$ and $g\overset{l}{\sim} h$ through the homotopies $H'\colon Cyl'(A)\to X$ and $H''\colon Cyl''(A)\to X$ respectively. We can take the maps $A\amalg A\to Cyl'(A)$ and $A\amalg A\to Cyl''(A)$ to be cofibrations from  Axiom~\emph{SM3}. Since cofibrations are stable under pushout from Lemma~\ref{pushoutclosed}, and $A\amalg A$ is obtained through the pushout
    \begin{center}
        \begin{tikzcd}
        0
        \arrow[r, "!"]
        \arrow[d, "!"]
        &A
        \arrow[d]
        \\
        A
        \arrow[r]
        & A\amalg A
        \end{tikzcd}
    \end{center}
    we have that the arrows $A\xrightarrow{i_1}Cyl'(A)$ and $A\xrightarrow{i_0}Cyl''(A)$ are cofibrations. Even further, from Axiom~\emph{SM1} these two maps are trivial cofibrations and $Cyl'(A)$ is cofibrant.

    We proceed to define $Cyl(A)$ as the pushout
    \begin{center}
        \begin{tikzcd}
        A
        \arrow[r, "i_1"]
        \arrow[d, "i_0"]
        & Cyl'(A)
        \arrow[d]
        \arrow[ddr, bend left, "\sim"]
        \\
        Cyl''(A)
        \arrow[r]
        \arrow[rrd, bend right, "\sim"]
        &Cyl(A)
        \arrow[rd, dotted]
        \\
        &&A
        \end{tikzcd}
    \end{center}
    From Lemma~\ref{pushoutclosed} the right vertical arrow is a trivial cofibration and therefore, $A\amalg A\to Cyl(A)\xrightarrow{\sim}A$ is a cylinder object. A further application of the universal property of $Cyl(A)$ for the homotopies $H'$ and $H''$ gives a homotopy $H\colon Cyl(A)\to X$ between $f$ and $h$.
    \item This can be proven analogously to the equivalent result for model categories \cite[Lemma 4.9]{DS}. This proof depends just on Axioms MC4(i) and MC5(i) for model-structures. In our case, it suffices to use \emph{SM2} and \emph{SM3} in Definition\ref{semi-model} instead.
    %\item 
    \item This follows from Lemma~\ref{liftingreplacement} applied to any two cofibrant replacements of $A$.
\end{enumerate}
\end{proof}
\begin{lem}\label{liftingreplacement}
Let $X$ and $Y$ be objects in a semi-model category $\mathcal{C}$. Let as well $QX\to X$ and $QY\to Y$ be cofibrant replacements.

Given a map $f\colon X\to Y$ in $\mathcal{C}$  there is a map $\hat{f}\colon QX\to QY$ that lifts $f$. The map $\hat{f}$ depends up to left homotopy only on $f$, and is a weak equivalence if and only if so is $f$.
\end{lem}
\begin{proof}
Existence of such $\hat{f}\colon QX\to QY$ is given by finding a lifting of $f\circ p_X$ in the diagram
\begin{center}
\begin{tikzcd}
\emptyset
\arrow[r]
\arrow[d]
& QY
\arrow[d, "\sim", swap, ""]
\\
QX
\arrow[r, "f\circ p_X"]
\arrow[ru, dotted]
& Y
\end{tikzcd}
\end{center}

This lifting exists in a semi-model category from Axiom \emph{SM2} of Definition~\ref{semi-model}.

The unicity of the homotopy class of $\hat{f}$ follows from Proposition~\ref{machinerysemimodel}.\ref{2}.
\end{proof}

\begin{lem}\label{Icellpushout}
In a semi-model category the maps in $Cof$ are exactly the maps with the left lifting property with respect to $\mathcal{W}\cap Fib$.

As well, the maps with cofibrant domain that have the left lifting property with respect to $Fib$ are exactly those that are already in $\mathcal{W}\cap Cof$.
\end{lem}

The proof of this lemma is the same as for \cite[Proposition 3.13 (i) and (ii)]{DS}.

Since trivial fibrations determine the cofibrations and respectively, fibrations determine trivial cofibrations with cofibrant domain, we can use the argument of the proof of \cite[Proposition 3.14 (i) and (ii)]{DS}. Therefore, we have the following lemma:

\begin{lem}\label{pushoutclosed}
In a semi-model category, $Cof$ is closed under pushouts. Similarly, $\mathcal{W}\cap Cof$ if closed under pushouts along morphisms between cofibrant objects. 
\end{lem}

\bibliographystyle{plain}
\bibliography{semimodelstrucholierinehart}
\end{document}